\documentclass{jgaa-art}

\usepackage{graphicx}
\usepackage{amsmath}
\usepackage{amssymb}
\usepackage{dsfont}
\usepackage{complexity}
\usepackage[T1]{fontenc}
\usepackage[utf8]{inputenc}
\usepackage{enumerate}

\newtheorem{definition}{Definition}
\newtheorem{corollary}{Corollary}
\newtheorem{proposition}{Proposition}
\newtheorem{observation}{Observation}

\newenvironment{proofof}[1]{\par\addvspace\topsep\noindent
	{\bf Proof #1:} \ignorespaces }{\qed}

\graphicspath{{figures/}}
\newcommand{\fig}[1]{\figurename~\ref{#1}}

\def\RR{{\mathds{R}}}
\def\PP{{\mathds{P}}}
\def\AA{{\mathcal{A}}}
\DeclareMathOperator{\conv}{conv}

\usepackage{todonotes}

\begin{document}

\doi{}
\Issue{0}{0}{0}{0}{0} 
\HeadingAuthor{O. Aichholzer et al.} 
\HeadingTitle{Minimal Representations of Order Types} 
\title{Minimal Representations of Order Types by Geometric Graphs}
\Ack{Research %
	supported by the German Science Foundation (DFG), the
	Austrian Science Fund (FWF), and the Swiss National Science 
	Foundation (SNSF) within the collaborative DACH project 
	\emph{Arrangements and Drawings}.
	O.A., I.P., and B.V.\ were supported by %
	Austrian Science Fund 
	(FWF) grant W1230. 
	M.B., J.K., and P.V.\ were supported by %
	grant no.~18-19158S of the 
	Czech Science Foundation (GA\v{C}R). 
	M.B.\ and J.K.\ were %
	supported by Charles University project UNCE/SCI/004. 
	M.B.\ has received funding from European Research Council (ERC) 
	under the European Union's Horizon 2020 research. 
	M.H.\ and E.W.\ were supported by SNSF Project 200021E-171681.
	A.P.\ was supported by a Schr\"odinger fellowship of the 
	Austrian Science Fund (FWF): J-3847-N35. 
	M.S.\ was partially supported by DFG Grant FE 340/12-1.
	W.M.\ was partially supported by ERC StG 757609 and DFG Grant
	3501/3-1.}

\author[1]{Oswin Aichholzer}{oaich@ist.tugraz.at}
\author[2]{Martin Balko}{balko@kam.mff.cuni.cz}
\author[3]{Michael Hoffmann}{hoffmann@inf.ethz.ch}
\author[2]{Jan Kynčl}{kyncl@kam.mff.cuni.cz}
\author[4]{Wolfgang Mulzer}{mulzer@inf.fu-berlin.de}
\author[5]{Irene Parada}{i.m.de.parada.munoz@tue.nl}
\author[1]{Alexander Pilz}{apilz@ist.tugraz.at}
\author[6]{Manfred Scheucher}{scheucher@math.tu-berlin.de}
\author[2]{Pavel Valtr}{valtr@kam.mff.cuni.cz}
\author[1]{Birgit Vogtenhuber}{bvogt@ist.tugraz.at}
\author[3]{Emo Welzl}{emo@inf.ethz.ch}

\affiliation[1]{Institute of Software Technology, Graz University of Technology, Austria}

\affiliation[2]{Institute for Theoretical Computer Science (CE-ITI), Charles University, Prague, Czech Republic}

\affiliation[3]{Department of Computer Science, ETH Z\"urich, Switzerland}

\affiliation[4]{Institut f\"ur Informatik, Freie Universit\"at Berlin, Germany}

\affiliation[5]{Department of Mathematics and Computer Science, TU Eindhoven,\\ The Netherlands}

\affiliation[6]{Institute of Mathematics, Technische Universit\"at Berlin, Germany}


\maketitle

\begin{abstract}
In order to have a compact visualization of the order type of 
a given point set $S$, 
we are interested in geometric graphs on $S$ with few edges that unambiguously display %
the order type of $S$. 
We introduce the concept of \emph{exit edges}, 
which prevent the order type from changing under continuous motion of vertices.
That is, 
in the geometric graph on $S$ whose edges are the exit edges, 
in order to change the order type of $S$, at least one vertex needs to move across an exit edge.
Exit edges have a natural dual characterization,  
which allows us to efficiently compute them and to bound their number. 
\end{abstract}

\Body 

\section{Introduction} 

Let $S, T \subset \RR^2$ be two sets of $n$ labeled points
in general position, 
that is, such that no three points in a set are collinear. 
We say that $S$ and $T$ have \emph{the same order type} if there is a 
bijection $\varphi: S \rightarrow T$ such that any triple $(p, q, r) \in S^3$ of three
distinct points has the same orientation (clockwise or
counterclockwise) %
as the image 
$(\varphi(p), \varphi(q), \varphi(r)) \in T^3$.
The resulting equivalence relation on planar $n$-point sets has 
a finite number of equivalence classes, the \emph{order types}~\cite{multidimensional_sorting}.
Representatives of all the distinct order types of five and six points are illustrated in \fig{fig:5ot}.
Among other things, the order type determines
which geometric graphs can be drawn on a point set without crossings.
Thus, order types %
appear ubiquitously in the study of
extremal problems on geometric graphs.

\begin{figure}[h]
	\centering
	\includegraphics{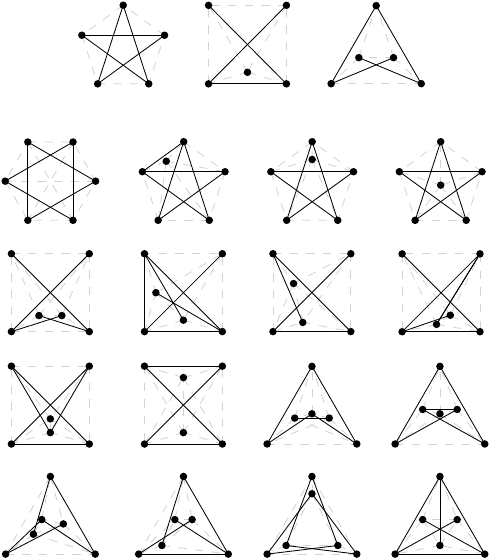}
	\caption{Representatives of the three order types of five points 
		and the sixteen order types of six points
		in general position. 
		Exit edges are drawn in black.}\label{fig:5ot}
\end{figure}

Now, suppose we have found that an order type is interesting for a problem, and 
we would like to illustrate it in a publication.
One solution is to give explicit coordinates of a representative 
point set~$S$; see \fig{fig:exit_edge} (left). 
This is unlikely to satisfy most readers.
We could also present $S$ as a set of dots in a figure.
For some point sets (particularly those with extremal 
properties), the reader may find it difficult to discern the 
orientation of an almost collinear point triple.
To mend this, we could draw all lines spanned by two 
points in $S$.
In fact, it suffices to present only the \emph{segments} between 
the point pairs (the complete geometric graph on $S$).
The orientation of a triple can then be obtained
by inspecting the corresponding triangle; see \fig{fig:exit_edge} (middle).
However, such a drawing is rather dense,
and we may have trouble following an edge from one endpoint to the other.
Therefore, we want to reduce the number of edges in the drawing as
much as possible, but so that the order type remains uniquely identifiable. 
In \fig{fig:exit_edge} (right)  
the triple orientations are unambiguously displayed since continuous deformations that keep the edges straight do not allow to change the orientation of any triple.

\begin{figure}[htb]
	\hfil
	\begin{minipage}[b]{.32\linewidth}\footnotesize
		\texttt{%
			(-1,1)\\
			(1,1)\\
			(-1,-1)\\
			(1,-1)\\
			(-0.6,0.4)\\
			(-0.6,-0.4)}
		\label{fig:exit_edge:0} 
	\end{minipage}
	\hfil\hfil
	\begin{minipage}[b]{.32\linewidth}
		\centering\includegraphics[page=1]{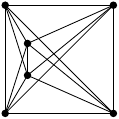}
		\label{fig:exit_edge:1} 
	\end{minipage}
	\hfil\hfil
	\begin{minipage}[b]{.32\linewidth}
		\centering\includegraphics[page=2]{exit_edge_basics}
		\label{fig:exit_edge:2} 
	\end{minipage}
	\hfil
	\caption{
		Three different representations of an order type of six points.
	}\label{fig:exit_edge}
\end{figure}
\paragraph{Results} 
We introduce the concept of \emph{exit edges} to
capture which edges are sufficient to uniquely identify a given order
type in a robust way under continuous motion of vertices. 
\emph{Exit graphs}, defined as the geometric graphs whose edges are the exit edges, 
are \emph{supporting} for a point set: 
in an exit graph at least one vertex needs to move across an (exit) edge in
order to change the order type. 
(For precise definitions of these concepts we refer to Definitions~\ref{def:supporting} and \ref{def:exit-edge}.)
Though exit edges are defined on a point set, 
the set of exit edges only depends on the order type and not on the particular representative. 

We give an alternative characterization of exit
edges in terms of the dual line arrangement, where an exit edge  
corresponds to one or two empty triangular cells. 
This allows us to efficiently compute the set of exit
edges for a given set of $n$ points %
in $O(n^2)$ time and space.

Using the more general framework of abstract order types and their dual
pseudoline arrangements, we prove that every set of $n\ge 4$ points has
at least $(3n-7)/5$ exit edges. 
We also describe a family of $n$
points with $n-3$ exit edges, showing that the best possible lower bound is of order  $\Omega(n)$. 
An upper bound of $n(n-1)/3$ follows
from known results on the number of triangular cells in line
arrangements~\cite{Gruenbaum_aas}. 
Thus, compared to the complete
geometric graph with $n(n-1)/2$ edges, using only exit edges saves  
at least one third of the edges. 
We present a random construction with a quadratic expected number of exit edges. 

Exit graphs are not always minimal supporting graphs. 
In particular, the requirement of keeping the edges straight together with the non-stretchability of certain pseudoline arrangements 
can result in exit edges being sometimes unnecessary.
The relation between the number of exit edges and the minimum number of edges in a supporting geometric graph is an open question.

\paragraph{Identification of order types} 
Let $S$ be a set of $n$ labeled points in the plane.
A \emph{geometric graph} on $S$ is a graph with vertex set $S$ 
whose edges are line segments between their
endpoints. 
A geometric graph is thus a drawing of an abstract graph.
Two geometric graphs $G$ and $H$ are \emph{isomorphic} if there is an orientation-preserving homeomorphism of the plane transforming 
$G$ into~$H$. 
Each class of this equivalence relation may be described 
combinatorially by the cyclic orders of the edge segments 
around vertices and crossings, and by the incidences of 
vertices, crossings, edge segments, and faces.
In the following, we will consider topology-preserving deformations.  
An \emph{ambient isotopy} of the Euclidean plane is a continuous 
map $f: \RR^2\times [0,1] \rightarrow \RR^2$ such that 
$f(\cdot, t)$ is a homeomorphism for every $t\in [0,1]$ 
and $f(\cdot, 0) = \text{Id}$. 
Note that if there is an ambient isotopy transforming a geometric graph $G$ into another geometric graph $H$, then no vertex can cross through an edge and $G$ and $H$ are isomorphic. 
\fig{fig:isotopy} shows an illustration. 

\begin{figure}[th]
	\centering
	\includegraphics[page=5]{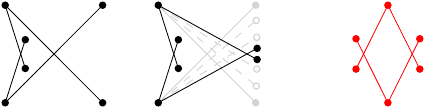}
	\caption{
		The geometric graph on the left can be transformed by an ambient isotopy into the geometric graph in the middle, but not into the geometric graph on the right. 
	}\label{fig:isotopy} 
\end{figure}

\begin{definition}\label{def:supporting}
	Let $G$ be a geometric graph on a point set $S$. %
	We say that $G$ %
	is \emph{supporting} for $S$ if every ambient isotopy $f$ of $\RR^2$ 
	that, for every $t\in[0,1]$, 
	keeps the images of the edges of $G$ straight 
	(thus, transforming $G$ into another geometric graph) and
	allows 
	at most one triple of collinear points of $f(S,t)$    
	also preserves the order type of the vertex set.
\end{definition}

Clearly, every complete geometric graph is supporting  
since all the triangles preserve their orientation, 
but there are supporting graphs with fewer edges, like the one in \fig{fig:isotopy} (left).

\paragraph{Related work} 
The connection between order types and geometric graphs 
has been studied intensively, both for
planar drawings and for drawings minimizing the number of crossings.
For example, it is \NP-complete to decide whether a planar 
graph can be embedded on a given point set~\cite{cabello}.
Continuous movements of the 
vertices of plane geometric graphs have also been considered~\cite{morph}.
The continuous movement of points maintaining the order type 
was considered by Mn\"ev~\cite{FelsnerGoodman2017,mnev}. 
He showed that there are point sets with the same order 
type such that there is no ambient isotopy between them 
preserving the order type,
settling a conjecture by Ringel~\cite{ringel}.
The orientations of triples that have to be fixed to determine the 
order type are strongly related to the concept of 
\emph{minimal reduced systems}~\cite{coordinatization}.
\emph{Compact} encodings of order types using few bits and allowing for fast orientation queries have also been studied. 
Cardinal \emph{et~al.}~\cite{encodings19} presented such an encoding for order types of $n$ points 
that uses $O(n^2 (\log \log n)^2 / \log n)$ bits, 
while there are $2^{\Theta (n\log n)}$ order types.

\paragraph{Outline} 
We introduce the concept of \emph{exit edges} for a given point set. 
The resulting \emph{exit graphs} are always 
supporting, though they are not necessarily minimal. 
In Section~\ref{sec:characterization} we show that some exit edges are rendered unnecessary by 
non-stretchability of certain pseudoline arrangements. 
Despite being non-minimal in general, we argue that exit graphs 
are good candidates for supporting graphs by discussing their 
dual representation in pseudoline arrangements
(Section~\ref{sec:empty_triangles}).
This connection allows us to both compute exit edges efficiently and
give bounds on their number (Section~\ref{sec:bounds}). 
Supporting graphs in general need not be connected,
and two minimal geometric graphs that are supporting for point sets 
with different order types can be drawings of 
the same abstract graph; see \fig{fig:5ot} (right).
Thus, the structure of the drawing is crucial.
In Section~\ref{sec:properties} we provide some further properties of the exit graphs. 
We conjecture that geometric graphs whose edges are the exit edges are not only 
supporting but also they encode the order type, 
as discussed in Section~\ref{sec:conclusion}. 

\section{Exit edges}\label{sec:characterization}

To obtain a supporting graph with fewer edges than the complete geometric graph, 
we select edges so that no vertex of the resulting geometric graph can be continuously deformed (as in Definition~\ref{def:supporting}) to 
change the order type while preserving isomorphism.

\begin{definition}\label{def:exit-edge}
	Let $S \subset \RR^2$ be finite and in general position. 
	Let $a, b, c \in S$ be distinct.
	Then, $ab$ is an \emph{exit edge with witness} $c$ if 
	there is no $p \in S$ such that the line 
	$\overline{ap}$ separates $b$ 
	from $c$ or the line $\overline{bp}$ separates $a$ from $c$. 
	We say that $ab$ is an \emph{exit edge} if there exists a point $c$ such that $ab$ is an exit edge with witness~$c$. 
	The geometric graph on $S$ whose edges are all the exit edges 
	is called the \emph{exit graph} of $S$.
\end{definition}
Equivalently, $ab$ is an \emph{exit edge} with \emph{witness} 
$c$ if and only if 
the double-wedge through $a$ between $b$ and $c$ and the 
double-wedge through $b$ between $a$ and $c$ contain no point of $S$ 
in their interior; see \fig{fig:channel} (left). 
We note that the exit graph is invariant under nondegenerate affine transformations.

\begin{figure}[ht]
	\centering
	\includegraphics[page=3]{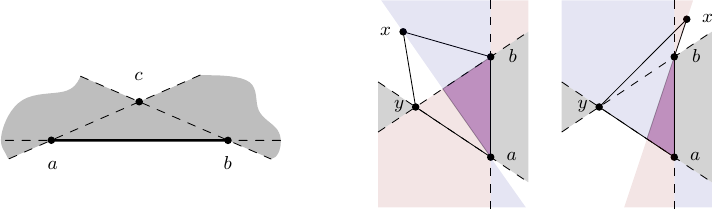}
	\caption{Characterizing exit edges. 
		Left: If the gray region is empty of points, 
		then the edge $ab$ is an exit edge.
		Right: An illustration of the proof of 
		Proposition~\ref{lem:characterisation_of_exit_edges}. 
	}\label{fig:channel} 
\end{figure}

An exit edge has at most two witnesses.
If $|S| \ge 4$ and $ab$ is an exit edge in $S$ with witness~$c$,
neither $ac$ nor $bc$ can be an exit edge with witness $b$ or $a$, 
respectively, 
as otherwise the union of empty regions would cover the rest of the whole plane except the points $a$, $b$, and $c$.
We illustrate the set of exit edges for sets of 5 points in \fig{fig:5ot} (top).

Exit edges can be characterized via 4-holes.
For an integer $k \ge 3$, a \emph{(general) $k$-hole} in $S$ 
is a simple polygon $\mathcal{P}$ spanned by $k$ points of $S$
whose interior contains no point of~$S$.
If $\mathcal{P}$ is convex, we call $\mathcal{P}$ a \emph{convex $k$-hole}.
A point $a \in S$ or an edge $ab$ of the complete geometric graph on $S$ 
is \emph{extremal} for $S$ 
if it lies on the boundary of the convex hull of~$S$.
A point or an edge that is not extremal in $S$ is 
\emph{internal} in $S$.

\begin{proposition}\label{lem:characterisation_of_exit_edges}
	Let $S \subset \RR^2$ be a point set in general position and 
	let $a, b\in~S$.
	Then, $ab$ is not an exit edge of $S$ if and only if 
	the following conditions hold:
	\begin{enumerate}
		\item
		If $ab$ is extremal for $S$, 
		then $ab$ is an edge of at least one convex 4-hole in~$S$.
		\item 
		If $ab$ is internal in $S$, 
		then there are two 4-holes $abxy$ and $bauv$, 
		in counterclockwise order,
		such that their reflex angles (if any) are incident to~$ab$.
	\end{enumerate}
\end{proposition}

\noindent We remark that an internal exit edge either has a witness 
on both sides or is incident to at least one (not necessarily convex) 4-hole on one side.

\begin{proof}
	Let $ab$ be an exit edge with a witness~$c$ 
	that lies, without loss of generality, to the left 
	of~$\overrightarrow{ab}$.
	Suppose  
	there is a general 4-hole $abxy$, traced counterclockwise, 
	such that the reflex angle of $abxy$ (if it exists) 
	is incident to~$ab$. 
	We can assume that $y$ lies to the left of~$\overrightarrow{ab}$,
	as in \fig{fig:channel} (right).	
	First, suppose that $abxy$ is convex (this must hold
	if $ab$ is extremal). 
	Since $ab$ is an exit edge with witness $c$,
	the line $\overline{ax}$ does not separate $c$ from $b$ and
	the line $\overline{by}$ does not separate $c$ from $a$.
	Thus, $c$ must be inside the $4$-hole $abxy$, which is impossible.
	Second, suppose that $abxy$ is not convex (then, $ab$ is
	internal), and $x$ is to the right of $\overrightarrow{ab}$.
	Since $ab$ is an exit edge with witness $c$,
	the line $\overline{bx}$ does not separate $a$ from $c$ 
	and the line $\overline{ay}$ does not separate $b$ from $c$,
	so $c$ lies inside the $4$-hole $abxy$, again 
	a contradiction.
	
	Conversely, assume that $ab$ is not an exit edge.
	First, let $ab$ be extremal, and
	let $p$ be the closest point in 
	$S\setminus \{a,b\}$ to the line $\overline{ab}$.
	The triangle $abp$ is a $3$-hole in~$S$.
	Since $p$ is not a witness for $ab$, there is a point 
	$q \in S\setminus \{a,b,p\}$ such that, without loss of 
	generality, the line $\overline{bq}$ separates $a$ from $p$.
	Since $ab$ is extremal, $q$ lies on the same side of 
	$\overrightarrow{ab}$ as $p$ and, in particular, the polygon 
	$abpq$ is convex.
	If we choose $q$ so that it is the closest such point 
	to the line $\overline{ap}$, 
	the triangles $bpq$ and $abq$ are $3$-holes in $S$.
	Altogether, we obtain a convex $4$-hole $abpq$ in $S$.
	
	Second, let $ab$ be internal.
	Let $p$ be closest in $S\setminus \{a,b\}$ to 
	the line $\overline{ab}$ such that $p$ lies to the left of 
	$\overrightarrow{ab}$.
	The triangle $abp$ is a $3$-hole in $S$.
	Since $p$ is not a witness for $ab$, there is a 
	point $q \in S\setminus \{a,b,p\}$ such that either the 
	line $\overline{bq}$ separates $a$ from $p$ or 
	the line $\overline{aq}$ separates $b$ from $p$.
	If $q$ lies to the left of $\overrightarrow{ab}$, 
	we obtain a convex $4$-hole as in the previous case.
	Thus, we can assume that all such points $q$ 
	lie to the right of $\overrightarrow{ab}$.
	We choose the point $q$ so that it is (one of the) 
	closest to the line $\overline{ab}$ among all points that 
	prevent $ab$ from being an exit edge with witness $p$.
	Without loss of generality, we assume that the line 
	$\overline{bq}$ separates $a$ from~$p$.
	The choice of $q$ guarantees that $bpq$ is a $3$-hole in $S$.
	Thus, $abqp$ is a $4$-hole in $S$ incident to $ab$ from the left.
	An analogous argument with a point $p'$ from 
	$ S\setminus \{a,b\}$ that is closest to $\overline{ab}$ 
	such that $p'$ lies to the right of $\overrightarrow{ab}$ 
	shows that there is an appropriate $4$-hole in $S$ 
	incident to $ab$ from the right.
\end{proof}

\begin{proposition}\label{prop:exit_edge_continuous_movement}
	Let $S \subset \RR^2$ be finite and in general position and, 
	for every $t\in [0,1]$, let $S(t)$ be a 
	continuous deformation of $S$ at time $t$.  
	More formally, let $f: \RR^2\times [0,1] \rightarrow \RR^2$ 
	be an ambient isotopy and $S(t) = 
	\{f(s,t) \mid s\in S\}$, for $t \in [0, 1]$.
	Suppose that for every $t \in [0, 1]$, there is at most one collinear triple of points in $S(t)$. Let $(a,b,c)$ be the first triple
	to become collinear, at time $t_0 > 0$.
	If $c$ lies on the segment $ab$ in $S(t_0)$, then $ab$ is an 
	exit edge of $S(0)$ with witness~$c$.
\end{proposition}
\begin{proof}
	For $t \in [0, t_0)$, the triple orientations in $S(t)$ remain unchanged, 
	and in $S(t_0)$, the point $c$ lies on $ab$ and the orientations of all triples except $(a,b,c)$ are still unchanged. 
	Thus, for $t \in [0, t_0)$, there is no line through two points of $S(t)$ 
	that strictly separates the relative interior of $ab$ from $c$.
	In particular, there is no such separating line through $a$ or $b$ in $S(0)$.
	Hence, $ab$ is an exit edge with witness~$c$.
\end{proof}

\begin{corollary}
	The exit graph of every point set is supporting.
\end{corollary}

A line separates $c$ from the relative 
interior of $ab$ if and only if there is such a separating line through $a$ or $b$.
This may suggest that the exit edges are necessary for 
a supporting graph.
However, this is not true in general.
For example, in \fig{fig:not_nec} (left), we see 
a construction by Ringel~\cite{ringel}: 
$ab$ is an exit edge with
witness $c$, but $c$ cannot move over $ab$ 
without violating Pappus' theorem. 
In this situation, we might consider 
the \emph{abstract order type} for the 
triple orientations we would obtain after moving~$c$ over $ab$. 
Since there is no planar point set with 
this set of triple orientations, 
this abstract order type is not \emph{realizable}. 
Deciding realizability is (polynomial-time-)equivalent 
to the existential theory of the reals~\cite{mnev}. 
We will revisit these concepts in Section~\ref{sec:bounds}.
\begin{figure}[bth]
	\centering	
	\includegraphics[page=1]{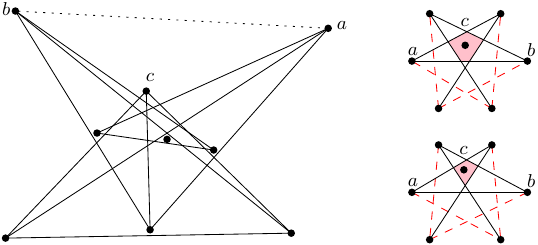}
	\caption{Left: moving $c$ over $ab$ to orient $(a,b,c)$ 
		clockwise, without changing the orientation of other 
		triples, would contradict Pappus's theorem~\cite{ringel}.
		Right: it is not always possible to move a witness $c$ continuously to the corresponding exit edge $ab$.
	}\label{fig:not_nec}
\end{figure}

We note that there are point sets where two or more 
other exit edges prevent a witness $c$ from crossing its corresponding exit 
edge $ab$; see, for example, \fig{fig:not_nec} (bottom right). 
Since the two geometric graphs in \fig{fig:not_nec} (right) are not isomorphic, they cannot be transformed into each other by a continuous deformation as the one used in Definition~\ref{def:supporting}.
However, in this example, 
while $c$ cannot move to $ab$ without changing the order 
type in \fig{fig:not_nec} (bottom right), if $ab$ were not
present, we could first change the point 
set to the one in \fig{fig:not_nec} (top right) and then move $c$ over $ab$.
Thus, $ab$ indeed has to be in a supporting graph.

\section{Exit edges and empty triangular cells}\label{sec:empty_triangles}

The (real) \emph{projective plane} $\PP^2$ is a 
non-orientable surface obtained by augmenting the Euclidean plane $\RR^2$ 
by a \emph{line at infinity}. 
This line has one \emph{point at infinity} for each direction, where
all parallel lines with this direction intersect. Thus, 
in $\PP^2$, each pair of parallel lines intersects in a unique point.

For a point set $S$ in the Euclidean plane, add a line 
$\ell_\infty$ to obtain the projective plane.
We use a duality transformation
that maps a point $s$ of~$\PP^2$ to a line $s^*$ in~$\PP^2$. 
In this way, we get a set of lines $S^*$ dual to~$S$,  
giving a projective line arrangement~$\AA$. 
The removal of a line from~$\AA$ does not disconnect~$\PP^2$. 
Since~$\PP^2$ has non-orientable genus 1, 
removing any two lines $\ell_1$ and $\ell_2$ from 
$\PP^2$ disconnects it into two components.
We call the closure of each of the two components a 
\emph{halfplane}\footnote{Here we follow the notation in~\cite{Gruenbaum_aas}. In the literature halfplanes are also called \emph{wedges}.} determined by $\ell_1$ and~$\ell_2$. 
The \emph{marked cell}~$c_{\infty}$ 
is the cell of~$\AA$ that contains the point $\ell_\infty^*$ 
dual to the line~$\ell_\infty$.
By appropriately choosing the duality transformation, 
we can assume that $\ell_\infty^*$ lies at vertical infinity.
We denote by $w(\ell_1,\ell_2)$ the halfplane
determined by $\ell_1$ and $\ell_2$ that does not contain the marked cell.

The combinatorial structure of~$\AA$, together with 
the marked cell, determines the order type of $S$.
We show how to identify exit edges and their 
witnesses in dual line arrangements.

We use the marked cell $c_{\infty}$ to 
orient the lines from~$S^*$: 
first, we orient the lines on the boundary of $c_{\infty}$
in one direction.
Then, we iteratively remove lines that have already been oriented,
and we define the orientation for the remaining lines from~$S^*$
by considering the new lines on the boundary of $c_{\infty}$.
Then, $c_{\infty}$ is the only cell
whose boundary is \emph{oriented consistently}, 
that is, it can be traversed completely
along the resulting orientation.
In particular,
for an unmarked triangular cell $\triangle$ in~$\AA$, 
the directed edges of $\triangle$ form a transitive order on its vertices, 
with a unique vertex of~$\triangle$
in the middle.
We call this vertex the \emph{exit vertex} of~$\triangle$
and the line through the other two vertices of~$\triangle$
the \emph{witness line} of~$\triangle$.

Note that if we consider the duality mapping a point $p=(p_x,p_y)$ from the real plane to 
the (non-vertical) line $p^*: y=p_xx - p_y$, 
then the described orientation procedure corresponds to orienting these dual lines from left to right.

Note that for two points $p,q\in S$ and their dual lines $p^*,q^*\in S^*$, 
$w(p^*,q^*)$ does not contain the marked cell and therefore 
its boundary is not oriented consistently.
 
The next theorem characterizes exit edges and their witnesses in the dual. 
In its proof we use the following property of projective duality:   
since it preserves incidences, 
the condition that no line spanned by two points of $S$ 
intersects the edge $pq$ 
is equivalent in $S^*$ to $w(p^*,q^*)$ not containing 
any vertex of~$\AA$.

\begin{theorem}\label{prop:duality_properties}
	Let $S \subset \RR^2$ be in general position,
	and let $a, b,c \in S$.
	Then, $ab$ is an exit edge with witness $c$ if and only if 
	the lines $a^*$, $b^*$, and $c^*$ 
	bound an unmarked triangular cell $\triangle$ in the 
	arrangement $\mathcal{A}$ of lines from $S^*$ so that $c^*$ 
	is the witness line of~$\triangle$
	and  the point $\overline{ab}^*=a^* \cap b^*$ 
	is the exit vertex of~$\triangle$.
\end{theorem}

\begin{proof}
	Let $\triangle$ be the triangular region determined by the intersection of the two halfplanes $w(a^*,c^*)$ and $w(b^*,c^*)$. %
	By the projective duality, $ab$ is an exit edge with witness $c$ in $S$ if and only if  
	no line of $S^*$ intersects $a^*$ inside $w(b^*,c^*)$ or $b^*$ inside $w(a^*,c^*)$. 
	In other words, if and only if two sides of $\triangle$, lying on $a^*$ and $b^*$, 
	contain no intersection with lines from~$S^*$. 
	This is equivalent to $\triangle$ being a cell of the arrangement $\mathcal{A}$. 
	Moreover, we can recognize $a^*$ and $b^*$ in $S^*$. 
	In the triangular cell~$\triangle$ that is the intersection of $w(a^*,c^*)$ and $w(b^*,c^*)$ the exit vertex is the intersection of	$a^*$ and $b^*$; see \fig{fig:triangle_lines}. 
	Consequently, the exit vertex $a^* \cap b^*$ is the dual of the line containing the exit edge~$ab$ (and vice versa). 
\end{proof}
\begin{figure}[htb]
	\centering
	\includegraphics[page=1]{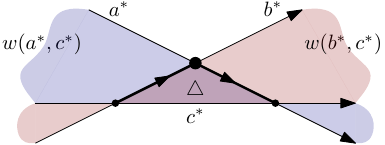}
	\caption{An illustration of the proof of Theorem~\ref{prop:duality_properties}. 
		If $ab$ is an exit edge with witness $c$ in~$S$, 
		then the two bold drawn segments of the corresponding triangular cell are unintersected, 
		and thus, bound an unmarked triangular cell in~$S^*$. 
		The exit vertex is represented with a black disk.}\label{fig:triangle_lines}
\end{figure}

Since line arrangements can be efficiently constructed in $O(n^2)$ time~\cite{arrangementsR2,arrangementsRd}, 
Theorem~\ref{prop:duality_properties} can be used to efficiently compute the set of exit edges. 

\begin{corollary}
	Let $S \subset \RR^2$ be a set of $n$ points in general position.
	Then the exit edges of $S$ can be enumerated in $O(n^2)$ time by constructing
	the dual line arrangement of $S$ and checking which cells are unmarked triangular cells. 
\end{corollary}

\section{On the number of exit edges}\label{sec:bounds}

Line arrangements can be generalized to so-called pseudoline arrangements.
A \emph{pseudoline} is a closed curve in the projective plane~$\PP^2$ 
whose removal does not disconnect $\PP^2$.
A set of pseudolines in~$\PP^2$,
where any two pseudolines cross exactly once,
determines a (projective) \emph{pseudoline arrangement}. 
If no three pseudolines intersect in a common point,
the pseudoline arrangement is \emph{simple}. 
All notions that we have introduced for line arrangements,
such as consistent orientations, exit vertices, or witness lines, 
naturally extend to pseudolines.

Two pseudoline arrangements are isomorphic if there is an isomorphism of the cell complexes into which they partition $\PP^2$. 
A pseudoline arrangement is \emph{stretchable} if it is 
isomorphic to a line arrangement, that is,
the corresponding cell complexes into which the
two arrangements partition $\PP^2$ 
are isomorphic.
Deciding if a pseudoline arrangement is stretchable is 
(polynomial-time-)equivalent to the existential theory of the 
reals~\cite{FelsnerGoodman2017,mnev}. 
The combinatorial dual analogues of line arrangements 
and pseudoline arrangements are order types and abstract order types,
respectively.

As a consequence of Theorem~\ref{prop:duality_properties},
the maximum number of \emph{triangular cells} in a simple 
projective pseudoline arrangement gives an upper bound on 
the number of exit edges of a point set. 
However, one triangular cell could be $c_\infty$, and there could 
be pairs of triangular cells with the same exit vertex.
We call a configuration of the latter type an \emph{hourglass}; 
see \fig{fig:hourglass}.
We say that the two pseudolines $p$ and $q$ 
that define the exit vertex of the two triangular cells of 
an hourglass $H$ \emph{slice}~$H$ and that $H$ is 
\emph{sliced} by $p$ and by~$q$.
\begin{figure}[tbh]
	\centering
	\hfill
	\begin{minipage}{.45\textwidth}
		\centering
		\includegraphics[page=1]{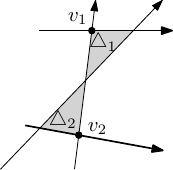}
		\label{fig:hourglass_1}  
	\end{minipage}
	\hfill
	\begin{minipage}{.45\textwidth}
		\centering
		\includegraphics[page=2]{hourglass}
		\label{fig:hourglass_2}  
	\end{minipage}
	\hfill
	
	\caption{
		Left: the two triangular cells $\triangle_1$ and $\triangle_2$ do 
		not form an hourglass, because they share a vertex that is not an exit vertex. 
		Right: the two triangular cells $\triangle_1$ and $\triangle_2$ 
		form an hourglass because they share an exit vertex. 
	}\label{fig:hourglass}
\end{figure}
\begin{observation}\label{obs:one_houglass}
	A triangular cell can be a part of at most one hourglass.
\end{observation}

\begin{observation}\label{obs:hourglasses}
	An exit edge $ab$ with two witness points is dual 
	to an hourglass with exit vertex $\overline{ab}^*$.
\end{observation}

Any projective arrangement of $n \ge 4$ lines has at least 
$n$ triangular cells, as each line is incident to at least 
three triangular cells~\cite{levi}. This is known to be tight.
Therefore, taking into account the marked cell $c_\infty$ and 
possible hourglasses, any set of $n \ge 4$ points has at least 
$\lceil \frac{n - 1}{2} \rceil$ exit edges.
We improve this lower bound by bounding from below the 
difference between the number of triangular cells and the number of 
hourglasses.

\begin{proposition}\label{prop:lower_bound}
	Any set of $n \ge 4$ points in the plane has at least 
	$(3n-7)/5$ exit edges.
\end{proposition}
For the proof of Proposition~\ref{prop:lower_bound} 
we use the following two lemmas.
The first is a theorem by Gr\"unbaum~\cite[Theorem~3.7 on p.~50]{Gruenbaum_aas}, 
and the second can be derived from the proof of that theorem.

\begin{lemma}[{Gr\"unbaum~\cite{Gruenbaum_aas}}]\label{lem:triang_pseudoline}
	In a simple pseudoline arrangement $L$ 
	every pseudoline from $L$ is incident to at least three triangular cells.
\end{lemma}

\begin{lemma}[{Gr\"unbaum~\cite{Gruenbaum_aas}}]\label{lem:triang_in_halfplane}
	Let $L$ be a simple arrangement of pseudolines, and 
	let $H$ be a closed halfplane determined by two pseudolines 
	$\ell_1, \ell_2 \in L$.
	If two other pseudolines of $L$ cross in the interior of $H$,
	then there is a triangular cell in $H$ that is incident to 
	$\ell_1$ but not to $\ell_2$.
\end{lemma}

\begin{proofof}{of Proposition~\ref{prop:lower_bound}}
	Let $L$ be a simple projective line arrangement of 
	$n\ge 4$ pseudolines $\ell_1, \ell_2, \dots, 
	\ell_n$.
	For each pseudoline $\ell_i \in L$, let $t_i$ be the number 
	of triangular cells incident to $\ell_i$ and $h_i$ the number 
	of hourglasses sliced by $\ell_i$. Set $x_i = t_i - h_i/2$.
	For each pseudoline $\ell_i \in L$, there 
	are three possible cases.
	
	\textbf{Case~(i):} there is no hourglass sliced by $\ell_i$.
	By Lemma~\ref{lem:triang_pseudoline}, every pseudoline 
	is incident to at least three triangular cells. 
	Thus, we have $x_i= t_i \ge 3$.
	
	\textbf{Case~(ii):} the pseudoline $\ell_i$ slices an hourglass 
	together with some pseudoline $\ell_j$ and 
	the interior of each of the two halfplanes determined by 
	$\ell_i$ and $\ell_j$ contains at least one crossing of some 
	other pair of pseudolines. By Lemma~\ref{lem:triang_in_halfplane}, 
	$\ell_i$ is incident to the two triangular cells of the hourglass
	plus at least two other triangular cells, one in each closed halfplane. 
	%
	Thus, $t_i\ge 4$. Observation~\ref{obs:one_houglass} 
	implies $h_i \le t_i/2$. Overall
	we get $x_i = t_i - h_i/2 \ge t_i - t_i/4 \ge (3/4) \cdot 4 = 3$.
	
	\textbf{Case~(iii):} the pseudoline $\ell_i$ slices an 
	hourglass together with some pseudoline $\ell_j$, and  
	one of the two closed halfplanes $H_1$ and $H_2$ 
	determined by $\ell_i$ and $\ell_j$ contains no crossing 
	of any other pair of pseudolines in its interior. 
	Suppose the closed halfplane that contains no further crossing 
	is $H_1$. Then, the hourglass sliced by $\ell_i$ and $\ell_j$ 
	is in $H_1$, as the other two lines defining the hourglass 
	do not cross in that halfplane; see \fig{fig:proof_lower}~(left). 
	Since $H_1$ contains no crossing in its interior, it is divided 
	by the other pseudolines into $4$-gons and the two triangular 
	cells of the hourglass. In particular, the marked cell is 
	bounded by at most four pseudolines, two of them being $\ell_i$ and $\ell_j$; 
	see \fig{fig:proof_lower} (right). Thus, there can be at most 
	four pseudolines for which case~(iii) applies. 
	Notice that in this case $h_i=1$, since any other hourglass sliced by $\ell_i$ 
	would have one triangular cell in each of the two halfplanes $H_1$ and $H_2$ 
	and the two triangular cells in $H_1$ form the already-counted hourglass 
	(and by Observation~\ref{obs:one_houglass} they cannot be part of another hourglass).
	Thus, we can only guarantee that 
	$x_i \ge 3 - 1/2 = 5/2$.  However, as we showed, 
	this case can happen for at most two pairs of pseudolines.

	\begin{figure}[tb]
		\centering
		\includegraphics[page=1]{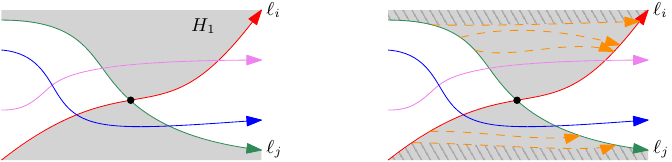}
		\caption{In case~(iii), both $\ell_1$ and $\ell_2$ 
			must bound the marked cell, shown 
			striped on the right picture. Moreover, 
			that cell is bounded by four pseudolines.}\label{fig:proof_lower}
	\end{figure}	
	
	Let $T$ be the total number of triangular cells in $L$
	and let $H$ be the total number of hourglasses. 
	Summing the contributions of cases (i)--(iii), we have
	\[
	3T - H 
	= \sum_{i = 1}^n t_i - \frac{1}{2} \sum_{i = 1}^n h_i
	= \sum_{i = 1}^n x_i 
	\geq 3\cdot (n-4) + 4 \cdot \left(\frac{5}{2}\right)
	= 3n - 2.
	\]
	
	By Observation~\ref{obs:one_houglass}, we have $T\ge 2H$.
	Combining these inequalities,  we get
	\[
	T-H = \frac{3T-H + 2(T-2H)}{5} \ge \frac{3T-H}{5} \ge \frac{3n-2}{5}.
	\]
	By Theorem~\ref{prop:duality_properties}, the number of 
	exit edges in a point set is equal to the number of exit 
	vertices in its dual line arrangement. In general, 
	the number of exit vertices in a pseudoline arrangement 
	is bounded from below by $T-H-1$. 
	Therefore, there are at least $\frac{3}{5}n - \frac{7}{5}$ exit edges.
\end{proofof}

We do not know if the lower bound in 
Proposition~\ref{prop:lower_bound} is tight. 
The smallest number of exit edges we could achieve 
is $n - 3$ for $n \ge 9$; see 
\fig{fig:construction_n_minus_3}. 
We exhaustively checked the set of exit edges for all order types of up to $10$ points using the order type database~\cite{otdb} 
and obtained that this construction with $n-3$ exit edges is optimal for $n=9, 10$. 
Moreover, the order type represented in \fig{fig:construction_n_minus_3} (left) is the only order type of $9$ points that requires $6$ exit edges.

\begin{figure}[hbt]
	\centering
	\includegraphics[page=2]{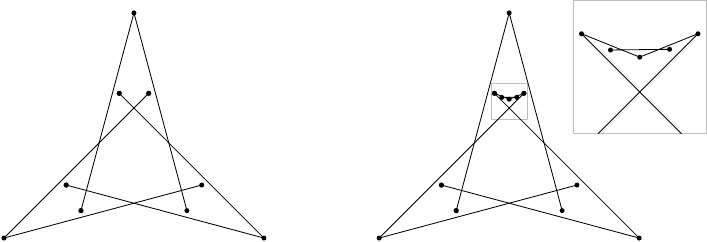}
	\caption{Construction with $n-3$ exit edges.}
	\label{fig:construction_n_minus_3}
\end{figure}

The number of triangular cells in a simple arrangement of $n$ 
lines in the projective plane $\PP^2$ 
is at most~$n(n-1)/3$~\cite{Gruenbaum_aas},
so there are at most  $n^2/3+O(n)$ exit edges. 
This means that 
representing an order type with the exit graph instead of 
the complete geometric graph saves at least one third of the edges.
Pal\'asti and F\"uredi~\cite{FuerdiPalasti84} showed that for every value of $n$ there is a simple arrangement of $n$ 
lines in $\PP^2$ with $n(n-3)/3$ triangular cells. 
Moreover, Roudneff~\cite{ROUDNEFF1986243} and Harborth~\cite{harborth} 
proved that the upper bound $n(n-1)/3$ is tight for infinitely 
many values of~$n$ (see also~\cite{Blanc2011}). 
The point sets that are dual to the currently-known arrangements that maximize the number of triangular cells 
have $n^2/6+O(n)$ exit edges, 
since most of their exit edges have two witnesses.
This gives a quadratic lower bound in the worst case, but the 
leading coefficient remains unknown.
It is worth noting that there are line arrangements 
with no pair of adjacent triangular cells~\cite{no_hourglass},
which implies the existence of point sets 
where every exit edge has precisely one witness. 

We now show a random construction with a quadratic expected number of exit edges. 

\begin{theorem}
	\label{thm-expectedExitEdges}
	Let $S = \{p_1,\ldots,p_n\}$ be a set of $n$ points in the plane with $p_i=(i,y_i)$ for every $i=1,\dots,n$, where each $y_i$ is chosen uniformly at random from the real interval~$[1,n]$.
	Then the expected number of exit edges in $S$ is~$\Theta(n^2)$.
\end{theorem}

The main idea of the proof of Theorem~\ref{thm-expectedExitEdges} is inspired by the proof of Theorem~2.3 from~\cite{BaranyFueredi1987}.

\begin{proof}
	The upper bound $O(n^2)$ on the number of exit edges in $S$ follows from the fact that the number of pairs of points from $S$ is $\binom{n}{2}$.
	In the rest of the proof we establish the lower bound $\Omega(n^2)$.
	
	First, note that all points of $S$ lie in the rectangle $R = [1,n] \times [1,n]$. Assume for convenience that $n$ is divisible by $5$.
	In the following, we identify each point $p_i$ with the number~$i$, which is the $x$-coordinate of $p_i$.
	Let $A = \{1,\ldots,\frac{n}{5}\}$, $B = \{\frac{2n}{5}+1,\ldots,\frac{3n}{5}\}$, and $C = \{\frac{4n}{5}+1,\ldots,n\}$. 
	Let $a$, $b$, and $c$ be fixed integers with $a\in A$, $b \in B$, and $c\in C$.
	We now find a lower bound on the probability that $p_ap_c$ is an exit edge of $S$ with witness $p_b$.
	
	The probability that the point $p_b$ has vertical distance at most~$1$ from the line segment $p_ap_c$ is at least~$\frac{1}{n}$, because the points from $\{b\} \times \mathbb{R}$ lying at distance at most $1$ from $p_ap_c$ form a vertical line segment of length $2$, and at least one half of this line segment is contained in $R$.
	
	In the following, we assume that $p_b$ has distance at most~1 from~$p_ap_c$.
	Consider a point $p_d$ with $d \in \{a+1,\dots,n\} \setminus \{b,c\}$.
	Since $a \in A$ and $b \in B$, we have $b-a \ge n/5$ and $d-a \le n$.
	Since $p_b$ has vertical distance at most 1 from~$p_ap_c$, the vertical side of the triangle $T$ bounded by the vertical line $\{b\} \times \mathbb{R}$ and by the rays $\overrightarrow{p_ap_b}$ and $\overrightarrow{p_ap_c}$ has length at most $1$;  see Figure~\ref{fig-expectedExitEdges}.
	Since the triangle~$T'$ bounded by these two rays and by the vertical line $\{d\} \times \mathbb{R}$ is similar to $T$, and since $d-a \le 5 (b-a)$, the vertical side of $T'$ has length at most $5$.
	Thus, the probability that $p_d$ lies in the convex wedge spanned by the rays $\overrightarrow{p_ap_b}$ and $\overrightarrow{p_ap_c}$ is at most~$5/n$.
	An analogous argument shows that the probability that a point~$p_d$ with $d \in \{1,\dots,c-1\} \setminus \{a,b\}$ lies in the convex wedge spanned by the rays $\overrightarrow{p_cp_a}$ and $\overrightarrow{p_cp_b}$ is at most~$5/n$.
	In total, the probability that $p_ap_c$ is an exit edge of the point set $\{p_a,p_b,p_c,p_d\}$ with witness $p_b$ is at least $1-10/n$.
	
	\begin{figure}
		\centering
		\includegraphics{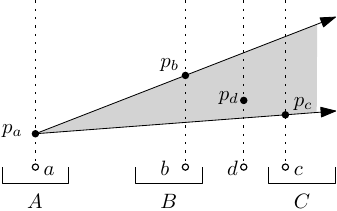}
		\caption{An illustration of the proof of Theorem~\ref{thm-expectedExitEdges}.}
		\label{fig-expectedExitEdges}
	\end{figure}
	
	Altogether, the probability that $p_ap_c$ is an exit edge of $S$ with witness~$p_b$ and that $p_b$ is at vertical distance at most~$1$ from $p_ap_c$ is at least
	\[
	\frac{1}{n} \cdot \prod_{d \in \{1,\dots,n\} \setminus \{a,b,c\}} \left(1-\frac{10}{n}\right) = 
	\frac{1}{n} \cdot \left(1-\frac{10}{n}\right)^{n-3} \ge \frac{1}{n\cdot e^{20}},
	\]
	where we use the inequality $1-x \ge e^{-2x}$ for every real $x$ with $0 \le x \le 1/2$.
	
	Since every exit edge of $S$ has at most two witnesses, the expected number of exit edges of $S$ is at least
	\[
	\frac{1}{2}\sum_{a \in A} \sum_{b \in B} \sum_{c \in C} \frac{1}{n\cdot e^{20}} \ge \Omega(n^2).
	\]
\end{proof}

Combining the point-line duality that maps a point $(a,b)$ to the line $\{(x,y)\in \mathbb{R}^2  \colon y = ax - b\}$ with Theorem~\ref{thm-expectedExitEdges}, we obtain the following result.

\begin{corollary}\label{cor-expectedExitEdges}
	Let $L = \{\ell_1,\ldots,\ell_n\}$ be a set of lines, where $\ell_i = \{(x,y) \in \mathbb{R}^2 \colon  y = i\cdot x - b_i\}$ and where $b_i$ is chosen uniformly at random from the real interval~$[1,n]$.
	Then the expected number of triangular cells in the line arrangement induced by~$L$ is~$\Theta(n^2)$.
\end{corollary}

\section{Properties of exit graphs}\label{sec:properties}

We present some further results on supporting graphs and exit graphs.

\begin{theorem}
	Any geometric graph supporting a point set $S \subset \RR^2$, with 
	$|S| \ge 9$, contains a crossing.
\end{theorem}
\begin{proof}
	Let $G$ be a geometric graph with vertex set $S$ without crossings.
	There is a point set $S'$ with 
	a different order type that also admits~$G$:
	Dujmovi{\'{c}}~\cite{untangling} showed that every plane graph admits a plane straight-line embedding with at least $\sqrt{n/2}$ points on a line; as we have a point set with a collinear triple that admits~$G$, there are at least two point sets in general position with a different order type that admit~$G$.
	Moreover, one can continuously morph $S$ to $S'$ while keeping the corresponding geometric graph planar and isomorphic to~$G$ (see, for example,~\cite{morph}).
	Therefore, $G$ does not support~$S$.
\end{proof}

\begin{proposition}\label{prop:extreme_points}
	Let $S$ be a point set in general position in $\RR^2$ and let $G$ be its exit graph.
	Every vertex in the unbounded face of $G$ is extremal, that is, it lies on the boundary of the convex hull of $S$.
\end{proposition}
Note that, as shown in \fig{fig:not_nec} (left), 
an analogous statement does not hold for general supporting graphs. 

\begin{proof}
	Suppose for contradiction that there is a point $p \in S$ incident to the unbounded face of the exit graph of $S$ and that is internal in $S$,
	that is, lies in the interior of the convex hull $\conv(S)$ of $S$.
	This means that there is a polygonal path inside $\conv(S)$ from~$p$ to the boundary of $\conv(S)$ such that the interior of this path intersects no exit edge of~$S$. Let $\delta(p)$ be the infimum of the lengths of such paths. Since $\conv(S)$ and $S$ are both compact sets, there is a polygonal path $P_p$ of length $\delta(p)>0$ from $p$ to the boundary of $\conv(S)$ that has no crossing with exit edges but may pass through other points of $S$.
	Among all such points $p$, let $r\in S$ be the point for which $\delta(r)$ is the minimum possible. Then $P_r$ is a single segment. Let $q$ be the endpoint of $P_r$ on the boundary of $\conv(S)$.
	
	If $q$ coincides with an extremal point in $S$, 
	we slightly perturb the point $q$ so that $q$ lies in the interior of an edge of $\conv(S)$ and 
	the line segment $rq$ does not intersect any exit edge of $S$.
	Let $s$ and~$t$ be the endpoints of the edge of $\conv(S)$ containing $q$; see \fig{fig:extremal_point} for an illustration.

	\begin{figure}[bth]
		\centering
		\includegraphics[page=1]{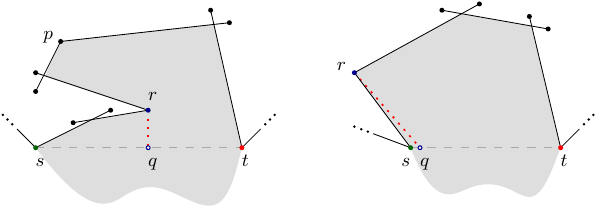}
		\caption{
			An illustration of the proof of 
			Proposition~\ref{prop:extreme_points}.
			The path between $r$ and $q$ is drawn as a red dotted line segment.}\label{fig:extremal_point}
	\end{figure}
	
	Since exit edges are invariant to 
	nondegenerate affine transformations 
	we assume without loss of generality that the following three conditions are satisfied.
	\begin{enumerate}[(i)]%
		\item The points $r$ and $q$ lie on the $y$-axis, $s$ has negative $x$-coordinate and $t$ has positive $x$-coordinate,
		\item the point $r$ lies above the line $\overline{st}$, and
		\item all points of $S$ have distinct $x$-coordinates. 
	\end{enumerate}
	
	To obtain a contradiction, we will show that the %
	segment $rq$ intersects the interior of an exit edge of~$S$.
	We will prove this in a dual setting.

	By applying the duality transformation mentioned in Section~\ref{sec:empty_triangles} that maps each point $p=(p_x,p_y)$ to the (non-vertical) line $p^*: y=p_xx - p_y$, we map the point set $S$ to the dual line arrangement~$S^*$.
	Due to the three conditions above, the lines $r^*$ and $q^*$ are horizontal and the lines $s^*$ and $t^*$ have a negative and a positive slope,
	respectively; see \fig{fig:extremal_dual}.
	By Theorem~\ref{prop:duality_properties}, a triple of points of $S$ representing the endpoints of an exit edge together with its witness, such that the $x$-coordinate of the witness is between the $x$-coordinates of the endpoints of the exit edge, corresponds to a triangular cell in $S^*$ where the dual of the witness is the line with median slope bounding this cell.
	\begin{figure}[htb]
		\centering
		\includegraphics[page=1]{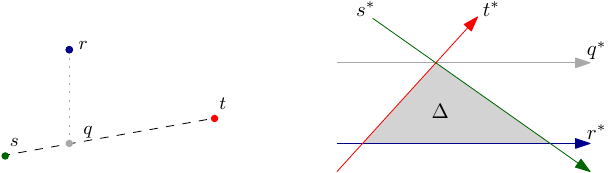}
		\caption{Applying the dual transformation to the point set $S$ (left) and obtaining the line arrangement $S^*$ (right).}\label{fig:extremal_dual}
	\end{figure}

	Let $\triangle$ be the triangular region bounded by the lines~$r^*$,~$s^*$, and~$t^*$. 
	Since the line segment $st$ is not an exit edge in~$S$, 
	the triangular region $\triangle$ is not a cell in~$S^*$. 
	Thus, the interior of $\triangle$ is intersected by some line from~$S^*$.
	Since $s$ and $t$ are vertices of $\conv(S)$, their duals $s^*$ and $t^*$ are incident to the upper envelope of $S^*$.
	
	Moving a point $p$ vertically down from $r$ to $q$ corresponds to sweeping the dual $S^*$ by a horizontal line $p^*$ from $r^*$ to $q^*$.
	Thus, meeting an exit edge of $S$ with $p$ corresponds to the situation in the dual in which the sweeping line $p^*$ meets a vertex of a triangular cell of~$S^*$ such that the vertex is an intersection of a line with a positive slope and a line with a negative slope.
	Therefore, the line segment $rq$ crosses an exit edge of $S$ if and only if there is a triangular cell~$\triangle'$ of $S^*$ between $r^*$ and  $q^*$ such that $\triangle'$ is bounded by a line with positive slope and a line with negative slope.
	To obtain a contradiction, we will show that $\triangle$ contains such a triangular cell $\triangle'$.
	
	We start with the line arrangement containing the lines $r^*$, $s^*$, and $t^*$.  
	First, we insert the set $L^+$ of lines from $S^*$ with positive slope that intersect the interior of~$\triangle$. 
	The goal is to find a triangular region $\triangle^+$ in $\triangle$ with one edge on~$s^*$ 
	such that no line from $S^*$ with positive slope intersects the interior of $\triangle^+$. 

	Since the lines $s^*$ and $t^*$ must bound the upper envelope (and are consecutive on it),
	no line from $S^*$ with positive slope can intersect $s^*$ above its intersection with $t^*$. 
	Thus, the lines from $L^+$ cannot intersect both $r^*$ and $t^*$ on the boundary of $\triangle$. 
	By definition, the lines from $L^+$ must intersect two of the segments bounding $\triangle$ and therefore   
	they must intersect $s^*$ on the boundary of~$\triangle$; 
	see \fig{fig:new_positive} (left).

	\begin{figure}
	\centering
	\includegraphics[page=3]{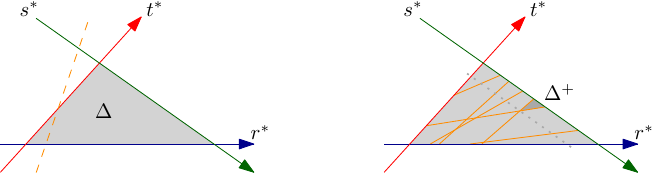}
	\caption{
		Inserting the set of lines $L^+$ from $S^*$ with positive slope that intersect the interior of $\triangle$. 
		Left: the dashed line cannot be in $L^+$ since the intersection of $s^*$ and $t^*$ must be on the upper envelope. Thus, the lines in $L^+$ must intersect $s^*$ on the boundary of $\triangle$.
		Right: finding a triangular region $\triangle^+$ inside $\triangle$ bounded by $s^*$.}
	\label{fig:new_positive}
	\end{figure}

	Consider the intersection point in $\triangle$ closest to $s^*$ 
	produced by two lines $\tilde{r}^*$ and $\tilde{t}^*$ (that possibly coincide with $r^*$ or $t^*$) from $\{r^*,t^*\}\cup L^+$. 
	We assume that the slope of $\tilde{t}^*$ is larger than the slope of $\tilde{r}^*$. 
	Since all the lines from $L^+$ intersect~$s^*$ on the boundary of $\triangle$, 
	the intersection of $\tilde{r}^*$ and $\tilde{t}^*$ is the leftmost vertex of a triangular cell $\triangle^+$ (of $\{r^*,s^*,t^*\}\cup L^+$) bounded by~$s^*$; 
	see \fig{fig:new_positive} (right) for an illustration.
	Moreover, $\triangle^+$ is contained in $\triangle$ 
	and it is thus a cell of the arrangement defined by $r^*$ and $s^*$  together with all the lines with positive slope from $S^*$ (including $t^*$ and all the lines in $L^+$).

	We now consider the lines from $S^*$ with negative slope. 
	We denote by $L^-$ the set of lines from $S^*$ with negative slope that intersect the interior of $\triangle^+$. 
	Analogously as before, we show that 
	there is a triangular cell $\triangle'$ of~$S^*$ inside~$\triangle^+$ 
	with one edge on $\tilde{t}^*$.
	
	Since the lines $s^*$ and $t^*$ must bound the upper envelope, 
	lines from $S^*$ with negative slope and steeper than $s^*$ 
	must intersect $s^*$ above its intersection with~$t^*$ (and therefore above its intersection with $\tilde{t}^*$).  
	Thus, the lines from $L^-$ cannot intersect both $\tilde{r}^*$ and $s^*$ on the boundary of~$\triangle^+$; 
	see \fig{fig:new_negative} (left).  
	By definition, the lines from $L^-$ must intersect two of the segments bounding~$\triangle^+$ and therefore   
	they must intersect $\tilde{t}^*$ on the boundary of~$\triangle^+$. 
	\begin{figure}[htb]
		\centering
		\includegraphics[page=4]{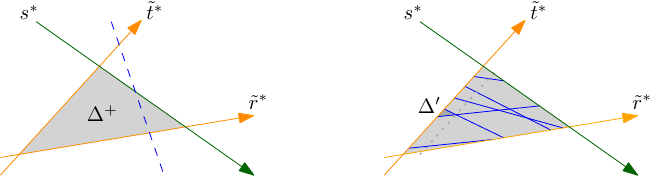}
		\caption{Inserting the set of lines $L^-$ from $S^*$ with negative slope that intersect the interior of $\triangle^+$. Left: the dashed line cannot be in $L^-$  since the intersection of $s^*$ and $\tilde{t}^*$ must be on the upper envelope. Thus, the lines in $L^-$ must intersect $\tilde{t}^*$ on the boundary of $\triangle^+$. Right: finding a triangular cell $\triangle'$ inside $\triangle^+$ bounded by $\tilde{t}^*$.}
		\label{fig:new_negative}
	\end{figure}
	
	In an analogous manner as before, 
	the intersection in $\triangle^+$ closest to $\tilde{t}^*$ defines a triangular cell $\triangle'$ inside $\triangle^+$ bounded by $\tilde{t}^*$;	see \fig{fig:new_negative} (right).  
	Thus, we found a triangular cell $\triangle'$ of $S^*$ contained in $\triangle$ 
	bounded by a line with positive slope and a line with negative slope.  
	Altogether, by duality, this implies that the %
	segment $rq$ crosses an exit edge of~$S$, which is a contradiction.
\end{proof}

\section{Concluding remarks}\label{sec:conclusion}

We conjecture that the geometric graph $G$ of exit edges not only is
supporting for $S$, %
but also that any point set $S'$ that is the vertex set of a geometric
graph isomorphic to~$G$ has the same order type as $S$. One might
conjecture that already knowing all exit edges and their witnesses (in
the dual line arrangement, all triangular cells and their orientations) is
sufficient to determine the order type. Surprisingly, this turns out to be
false.

A counterexample is sketched in \fig{fig:OTs_exist_triples} as a dual (stretchable) pseudoline arrangement of 14~lines in the projective plane, 
based on an example by 
Felsner and Weil~\cite{hbo-dcg}.
It consists of two arrangements of six lines in the Euclidean plane 
that are combinatorially different, 
but share the set of triangular cells and their orientations.
While the exit edges and their witnesses are the same for the two different order types, 
the corresponding exit graphs are not isomorphic.
\begin{figure}[htb]
	\centering
	\includegraphics[page=2,width=\textwidth]{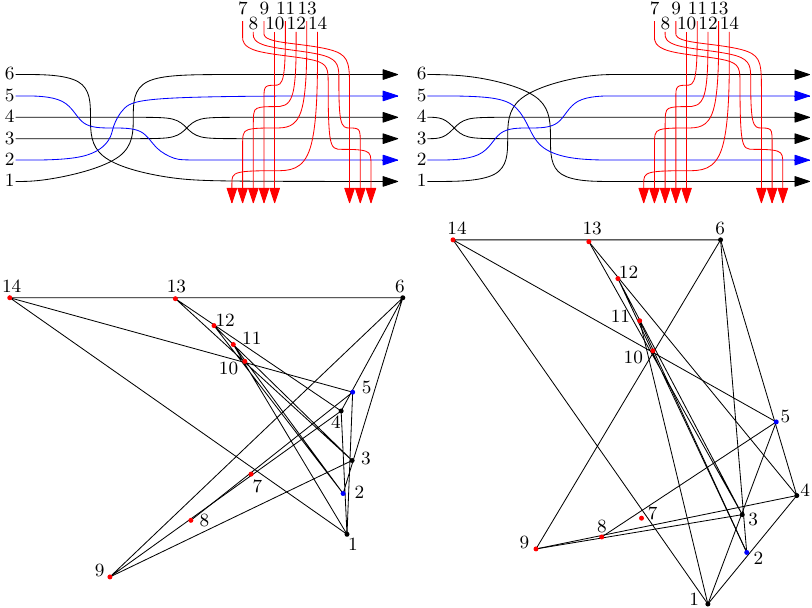}
	\caption{Top: 
		two arrangements of 14 pseudolines with the same set of triangular cells (extending~\cite[Figure~3]{hbo-dcg}).
		No triangular cell is crossed by the line at infinity.
		Bottom: 
		corresponding dual point sets and exit graphs.
		The order types are not the same (see for example the number of extremal points).}\label{fig:OTs_exist_triples}
\end{figure}

In the dual of that example the order of the triangular cells along each pseudoline differs, 
but that extra information is not enough to distinguish the two order types: 
We can modify the pseudoline arrangements in \fig{fig:OTs_exist_triples} by, essentially, 
duplicating pseudolines 1--6  
and making a pseudoline and its duplicate cross between the crossings with two red pseudolines (7--14).   
In \fig{fig:countereg} 
we present an illustration. 
It shows two pseudoline arrangements with the same triangular cells (including their orientations) 
and the same order of triangular cells along each pseudoline. 
However, the corresponding order types are not the same (see for example the number of extremal points). 
Note that the dual point sets of the pseudoline arrangements in \fig{fig:countereg} 
can be  obtained from the ones in \fig{fig:OTs_exist_triples} by adding a copy of points 1--6 close to the original respective points. 
Thus, we cannot reconstruct the order type from that information.

\begin{figure}[htb]
	\centering
	\includegraphics[page = 1]{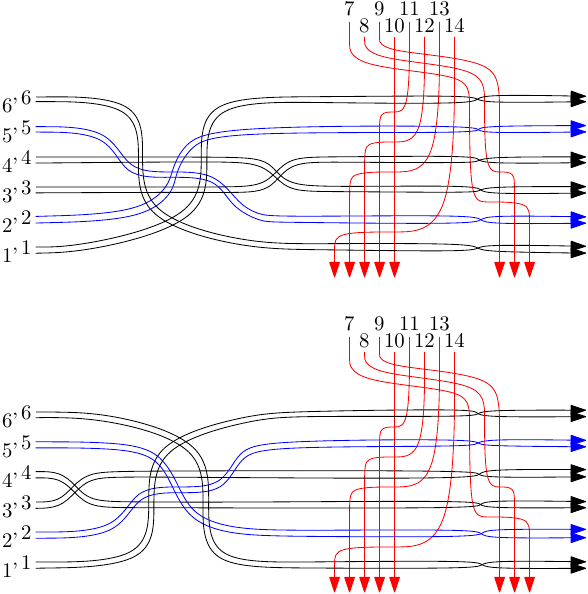}
	\caption{Two arrangements of 20 pseudolines with the same set of triangular cells (extending~\cite[Figure~3]{hbo-dcg}) and with the same ordering of the triangular cells along the pseudolines, but corresponding to different
		order types.}\label{fig:countereg}
\end{figure}

\section*{Acknowledgments} This work was initiated during the 
\emph{Workshop on Sidedness Queries}, October 2015, in Ratsch, Austria.
We thank Thomas Hackl, Vincent Kusters, and Pedro Ramos for valuable 
discussions.

\clearpage

\bibliographystyle{abbrvurl}
\bibliography{references}

\end{document}